\documentclass[reqno,11pt]{amsart}

\usepackage{amsmath, amssymb, amsthm}
\usepackage{geometry}
\usepackage{hyperref}
\usepackage{color}

\newtheorem{theorem}{Theorem}[section]
\newtheorem{lemma}[theorem]{Lemma}
\newtheorem{proposition}[theorem]{Proposition}
\newtheorem{definition}[theorem]{Definition}
\newtheorem{remark}[theorem]{Remark}

\numberwithin{equation}{section}

\newcommand{\norm}[1]{\|#1\|}
\newcommand{\grad}{\nabla}

\begin{document}

\title[Wave-MGT System with Logarithmic Nonlinearity]{Global Well-Posedness and Conditional Asymptotic Stability for a Coupled Wave-MGT System with Logarithmic Nonlinearity}
\author{Tae Gab Ha}

\address{Department of Mathematics and Institute of Pure and Applied Mathematics, Jeonbuk National University, Jeonju 54896, Republic of Korea}

\email{tgha@jbnu.ac.kr}



\maketitle

\begin{abstract}
We study a coupled system formed by a conservative wave equation and a dissipative Moore-Gibson-Thompson (MGT) equation on a bounded domain. The wave component is driven by the logarithmic source $f(u)=|u|^{\gamma-2}u\ln|u|$, $2<\gamma<\frac{2(n-1)}{n-2}$, and carries no direct damping. Rather than employing cross-multiplier arguments, we introduce the coupled variable $w=v+\tau v_{t}$, which reveals the exact energy structure associated with the interaction term. This formulation yields a genuine coupled energy together with a coercive quadratic form $Q_{\alpha}(u,w)=\norm{\nabla u}_{2}^{2}+\norm{\nabla w}_{2}^{2}+2\alpha(u,w)$, provided that $|\alpha|<\lambda_{1}$. Based on this structure, we construct a coupled potential well and prove global well-posedness of weak solutions for initial data lying below the corresponding well depth and inside the stable set. We also show that the energy is strictly dissipative through the MGT component. In addition, a modal analysis of the linearized system identifies a high frequency spectral obstruction to uniform exponential decay, quantifying the weakness of the dissipation transfer to the wave branch. Finally, assuming the relative compactness of the trajectory in the natural energy space and imposing $0<|\alpha|<\lambda_1$, we apply LaSalle's invariance principle to establish conditional asymptotic stability of the zero equilibrium. The result provides a structurally consistent indirect stabilization theorem for the coupled wave--MGT dynamics without relying on unjustified exponential decay claims.
\end{abstract}

\section{Introduction}

The mathematical analysis of evolution equations with logarithmic nonlinearities has attracted sustained attention because such terms arise naturally in mathematical physics and, at the same time, exhibit a delicate analytical structure. The logarithmic Schr\"odinger model introduced by Bia\l ynicki-Birula and Mycielski \cite{Bialynicki} is one of the classical starting points of this theory, while the foundational work of Cazenave and Haraux \cite{Cazenave} developed a rigorous PDE framework for evolution equations with logarithmic nonlinearities. In bounded domains, logarithmic sources of wave type have since been studied from several complementary viewpoints, including global existence, blow-up, and potential well analysis; see, for instance, \cite{Lian1, Lian2, AlGharabli}. In the present paper, the source term
\begin{equation*}
f(u)=|u|^{\gamma-2}u\ln|u|
\end{equation*}
is superquadratic for large amplitudes, changes sign near the origin, and fails to fit into the usual scale invariant templates of pure power nonlinearities. These features make the interplay between nonlinear growth and dissipation especially subtle.

Parallel to these developments, the Moore-Gibson-Thompson equation has emerged as a basic model in nonlinear acoustics and high intensity ultrasound. It is a third order in time evolution equation that incorporates thermal relaxation effects and avoids the paradox of infinite propagation speed present in classical Fourier based models; see, e.g., \cite{Kaltenbacher, Lasiecka, DellOro, Pellicer}. For the normalized MGT dynamics considered here, the dissipative regime is captured by the parameter restriction $b>\tau>0$, which isolates the physically relevant stable case.

When the wave and MGT components are coupled, the problem enters the realm of indirect stabilization: dissipation acts only on one equation, and the second component is stabilized solely through the coupling. This mechanism has been extensively studied in hyperbolic control theory, beginning with the work of Alabau-Boussouira and collaborators on weakly coupled systems; see \cite{Alabau1, Alabau2, Alabau3, Alabau4}. In related wave type systems, observability and stabilization are closely tied to geometric and compatibility conditions, as highlighted by the Bardos-Lebeau-Rauch theory \cite{Bardos} and by later results on indirectly damped systems with local or weak coupling \cite{Alabau5}. These works show that the transfer of damping from one component to another is highly sensitive to the structure of the coupling and, in general, one should not expect uniform exponential stabilization without additional observability mechanisms.

In this paper we study a linearly coupled system composed of a purely conservative wave equation and a dissipative MGT equation. The wave component is driven by the source term $f(u)=|u|^{\gamma-2}u\ln|u|$, subject to the strictly subcritical restriction
\begin{equation*}
2<\gamma<\frac{2(n-1)}{n-2}.
\end{equation*}
The MGT equation contributes dissipation only through the term $-b\Delta v_{t}$, whereas the wave equation carries no internal damping at all. Consequently, the total energy can decrease only through the MGT component and the zero order coupling. This is precisely the indirect stabilization scenario, but in a particularly delicate form: the wave component retains its own kinetic energy, the coupling is zero order, and the characteristic wave speed of the acoustic part ($c_{w}=1$) does not match the high-frequency speed naturally associated with the MGT dynamics ($c_{mgt}=\sqrt{b/\tau}$). A direct modal analysis of the linearized system (see Proposition~\ref{prop:modal-obstruction} below) shows that the wave branch eigenvalues satisfy
\begin{equation*}
\Re s_{k,\pm}
=
-\frac{\alpha^2}{2(b-\tau)\lambda_k^2}
+o(\lambda_k^{-2}),
\qquad \lambda_k\to\infty,
\end{equation*}
so the dissipation transferred to the wave component is only of order $\lambda_k^{-2}$ at high frequency. In particular, the linearized spectrum has no uniform gap from the imaginary axis. This quantifies the obstruction created by the speed mismatch and explains why the present paper is naturally led toward a qualitative asymptotic stabilization result rather than a uniform decay-rate statement.

The main structural observation of the present work is that the natural coupling variable is not $v$ itself, but rather
\begin{equation*}
w=v+\tau v_{t}.
\end{equation*}
This change of variables constitutes the key step of the analysis. It does not eliminate $v_{t}$ from the equations; instead, it reveals the correct conservative core of the coupled dynamics and leads to an exact energy identity in which the only dissipative contribution is $(b-\tau)\norm{\nabla v_{t}}_{2}^{2}$. The same variable also identifies the correct static pair $(u, w)$ for the potential well construction. In particular, the indefinite interaction term can be embedded into the joint quadratic form
\begin{equation*}
Q_{\alpha}(u,w)=\norm{\nabla u}_{2}^{2}+\norm{\nabla w}_{2}^{2}+2\alpha(u,w),
\end{equation*}
which remains coercive under the small-coupling condition $|\alpha|<\lambda_{1}$. This places the problem in a framework that is structurally consistent with the exact energy of the system, rather than with an artificially symmetrized multiplier ansatz.

Our first main result (Theorem \ref{thm:global_existence}) is a global well-posedness theorem in the coupled stable set. Starting from initial data below the coupled well depth and satisfying the strict positivity of the corresponding Nehari functional, we construct global weak solutions by a Galerkin approximation, prove the exact energy dissipation identity, and show that the stable region is positively invariant. A central technical point is the logarithmic estimate with $\epsilon$-absorption, which allows us to control the singular behavior of the logarithm near the origin without destroying the coercivity of the coupled quadratic form.

Our second main result (Theorem \ref{thm:asymptotic}) concerns the long time dynamics. Under the physically consistent hypothesis that the coupling is nontrivial ($0<|\alpha|<\lambda_{1}$), and assuming the relative compactness of the full trajectory in the natural phase space, we apply LaSalle's invariance principle \cite{LaSalle}. The exact dissipation law identifies the zero-dissipation set, and the coupled potential-well structure then shows that the only invariant state compatible with the subcritical energy level is the zero equilibrium. This yields the conditional asymptotic stability of the coupled wave-MGT system.

The rest of the paper is organized as follows. In Section 2, we introduce the transformed variable $w=v+\tau v_t$, derive the exact coupled energy identity, and record a linearized high frequency spectral obstruction to uniform exponential decay. Section 3 is devoted to the coupled potential well, including the coercivity of the quadratic form, the logarithmic $\epsilon$-absorption estimate, and the positivity of the well depth. In Section 4, we prove global well-posedness and uniqueness of weak solutions below the well depth. Finally, in Section 5, we establish the conditional asymptotic stability of the zero equilibrium by LaSalle's invariance principle.
\section{The model and the exact coupled energy}

Let $\Omega\subset\mathbb{R}^{n}$ ($n\ge 3$) be a bounded domain with smooth boundary. We write $\norm{\cdot}_{p}$ for the norm in $L^{p}(\Omega)$ and $(\cdot,\cdot)$ for the inner product in $L^{2}(\Omega)$. We study
\begin{equation}\label{eq:main}
\begin{cases}
u_{tt}-\Delta u+\alpha(v+\tau v_{t})=|u|^{\gamma-2}u\ln|u|, \quad &\text{in } \Omega\times(0,\infty), \\
\tau v_{ttt}+v_{tt}-\Delta v-b\Delta v_{t}+\alpha u=0, \quad &\text{in } \Omega\times(0,\infty), \\
u=v=0, \quad &\text{on } \partial\Omega\times(0,\infty), \\
u(\cdot,0)=u_{0}, \; u_{t}(\cdot,0)=u_{1}, \quad &\text{in } \Omega, \\
v(\cdot,0)=v_{0}, \; v_{t}(\cdot,0)=v_{1}, \; v_{tt}(\cdot,0)=v_{2}, \quad &\text{in } \Omega.
\end{cases}
\end{equation}
Throughout the paper we assume
\begin{equation}\label{eq:params}
b>\tau>0, \quad |\alpha|<\lambda_{1}, \quad 2<\gamma<\frac{2(n-1)}{n-2},
\end{equation}
where $\lambda_{1}$ denotes the first Dirichlet eigenvalue of $-\Delta$ on $\Omega$. We also define
\begin{equation*}
f(s)=|s|^{\gamma-2}s\ln|s|, \quad F(s)=\int_{0}^{s}f(\xi)d\xi=\frac{1}{\gamma}|s|^{\gamma}\ln|s|-\frac{1}{\gamma^{2}}|s|^{\gamma},
\end{equation*}
with the convention $F(0)=0$.

\subsection{The coupled variable $w=v+\tau v_{t}$}

Set $w=v+\tau v_{t}$. Then
\begin{equation*}
w_{t}=v_{t}+\tau v_{tt}, \quad w_{tt}=v_{tt}+\tau v_{ttt}.
\end{equation*}
Using the second equation in \eqref{eq:main} and the identity
\begin{equation*}
\Delta v+b\Delta v_{t}=\Delta(v+\tau v_{t})+(b-\tau)\Delta v_{t}=\Delta w+(b-\tau)\Delta v_{t},
\end{equation*}
we obtain the augmented system
\begin{equation}\label{eq:augmented}
\begin{cases}
u_{tt}-\Delta u+\alpha w=f(u), \\
w_{tt}-\Delta w-(b-\tau)\Delta v_{t}+\alpha u=0,
\end{cases}
\end{equation}
supplemented with the relation $w=v+\tau v_{t}$.

The point of \eqref{eq:augmented} is not to eliminate $v_{t}$, but to isolate the conservative coupling in the pair $(u, w)$ while keeping the MGT dissipation in the explicit term $-(b-\tau)\Delta v_{t}$.

\begin{proposition}[Exact coupled energy identity]\label{prop:energy_identity}
Assume that $(u, v)$ is a sufficiently smooth solution of \eqref{eq:main} and set $w=v+\tau v_{t}$. Define
\begin{align}\label{eq:energy_def}
\mathcal{E}(t) &:= \frac{1}{2}\norm{u_{t}(t)}_{2}^{2}+\frac{1}{2}\norm{w_{t}(t)}_{2}^{2}+\frac{\tau(b-\tau)}{2}\norm{\nabla v_{t}(t)}_{2}^{2} \nonumber \\
&\quad +\frac{1}{2}\norm{\nabla u(t)}_{2}^{2}+\frac{1}{2}\norm{\nabla w(t)}_{2}^{2}+\alpha(u(t),w(t))-\int_{\Omega}F(u(t))dx.
\end{align}
Then
\begin{equation}\label{eq:energy_dissipation}
\mathcal{E}(t)+(b-\tau)\int_{0}^{t}\norm{\nabla v_{s}(s)}_{2}^{2}ds=\mathcal{E}(0)
\end{equation}
for all $t\ge 0$. In particular, $\mathcal{E}^{\prime}(t)=-(b-\tau)\norm{\nabla v_{t}(t)}_{2}^{2}\le 0$.
\end{proposition}

\begin{proof}
Test the first equation in \eqref{eq:augmented} by $u_{t}$. Since
\begin{equation*}
(f(u),u_{t})=\frac{d}{dt}\int_{\Omega}F(u)dx,
\end{equation*}
we get
\begin{equation*}
\frac{d}{dt}\left(\frac{1}{2}\norm{u_{t}}_{2}^{2}+\frac{1}{2}\norm{\nabla u}_{2}^{2}-\int_{\Omega}F(u)dx\right)+\alpha(w,u_{t})=0.
\end{equation*}
Next test the second equation in \eqref{eq:augmented} by $w_{t}=v_{t}+\tau v_{tt}$. This yields
\begin{equation*}
\frac{d}{dt}\left(\frac{1}{2}\norm{w_{t}}_{2}^{2}+\frac{1}{2}\norm{\nabla w}_{2}^{2}\right)-(b-\tau)(\Delta v_{t},w_{t})+\alpha(u,w_{t})=0.
\end{equation*}
By integration by parts and the identity $w_{t}=v_{t}+\tau v_{tt}$,
\begin{equation*}
-(b-\tau)(\Delta v_{t},w_{t})=(b-\tau)(\nabla v_{t},\nabla w_{t})=(b-\tau)\norm{\nabla v_{t}}_{2}^{2}+\frac{\tau(b-\tau)}{2}\frac{d}{dt}\norm{\nabla v_{t}}_{2}^{2}.
\end{equation*}
Adding the two equalities and using
\begin{equation*}
\alpha(w,u_{t})+\alpha(u,w_{t})=\frac{d}{dt}(\alpha(u,w)),
\end{equation*}
we obtain
\begin{equation*}
\frac{d}{dt}\mathcal{E}(t)+(b-\tau)\norm{\nabla v_{t}}_{2}^{2}=0,
\end{equation*}
which integrates to \eqref{eq:energy_dissipation}.
\end{proof}

\subsection{A high frequency spectral obstruction to uniform exponential decay}

The previous proposition identifies the exact dissipative mechanism at the energy level. At the linearized level, one can also quantify how weakly this dissipation is transferred to the conservative wave component at high frequency. The purpose of this subsection is to identify a linearized high frequency obstruction that explains why the present work is formulated in terms of qualitative asymptotic stability rather than a uniform decay rate.

Consider the linearization of \eqref{eq:main} around the zero equilibrium:
\begin{equation}
\label{eq:linearized}
\begin{cases}
u_{tt}-\Delta u+\alpha(v+\tau v_t)=0,\\[1mm]
\tau v_{ttt}+v_{tt}-\Delta v-b\Delta v_t+\alpha u=0.
\end{cases}
\end{equation}
Let $\{e_k\}_{k\ge 1}$ be a Dirichlet eigenbasis of $-\Delta$ on $\Omega$, so that
\[
-\Delta e_k=\lambda_k e_k,
\qquad
0<\lambda_1\le \lambda_2\le \cdots,\qquad \lambda_k\to\infty.
\]
For modal solutions of the form
\[
u(x,t)=a_k e^{st}e_k(x),\qquad v(x,t)=c_k e^{st}e_k(x),
\]
the amplitudes $(a_k,c_k)$ satisfy
\[
\begin{pmatrix}
s^2+\lambda_k & \alpha(1+\tau s)\\
\alpha & \tau s^3+s^2+b\lambda_k s+\lambda_k
\end{pmatrix}
\binom{a_k}{c_k}=0.
\]
Hence the modal characteristic roots are exactly the zeros of
\begin{equation}
\label{eq:Pk}
P_k(s):=(s^2+\lambda_k)\bigl(\tau s^3+s^2+b\lambda_k s+\lambda_k\bigr)-\alpha^2(1+\tau s).
\end{equation}

\begin{proposition}[High frequency obstruction on the wave branch]
\label{prop:modal-obstruction}
Fix $b>\tau>0$ and $\alpha\neq 0$. For each sign $\sigma\in\{+1,-1\}$ there exists, for all sufficiently large $k$, a characteristic root $s_{k,\sigma}$ of $P_k$ such that
\begin{equation}
\label{eq:wave-branch-expansion}
s_{k,\sigma}
=
\sigma i\sqrt{\lambda_k}
-\frac{\alpha^2}{2(b-\tau)\lambda_k^2}
-\sigma i\,\frac{\alpha^2\tau}{2(b-\tau)\lambda_k^{3/2}}
+O(\lambda_k^{-5/2})
\qquad\text{as }k\to\infty.
\end{equation}
In particular,
\begin{equation}
\label{eq:wave-branch-real-part}
\Re s_{k,\sigma}
=
-\frac{\alpha^2}{2(b-\tau)\lambda_k^2}
+O(\lambda_k^{-5/2})
\longrightarrow 0^-.
\end{equation}
Equivalently, since $c_w=1$ and $c_{\mathrm{mgt}}=\sqrt{b/\tau}$,
\begin{equation}
\label{eq:wave-speed-gap}
\Re s_{k,\sigma}
=
-\frac{\alpha^2}{2\tau(c_{\mathrm{mgt}}^2-c_w^2)\lambda_k^2}
+O(\lambda_k^{-5/2}).
\end{equation}
Thus the transfer of dissipation from the MGT component to the wave branch is only of order $\lambda_k^{-2}$ at high frequency, and the linearized spectrum has no uniform gap from the imaginary axis.
\end{proposition}

\begin{proof}
Set $\omega_k:=\sqrt{\lambda_k}$ and fix $\sigma\in\{+1,-1\}$. We seek a root of $P_k$ near the uncoupled wave value $\sigma i\omega_k$. A direct evaluation of \eqref{eq:Pk} gives
\[
P_k(\sigma i\omega_k)=-\alpha^2(1+\sigma i\tau\omega_k),
\]
while differentiation yields
\[
P_k'(\sigma i\omega_k)=-2(b-\tau)\omega_k^4-\alpha^2\tau,
\qquad
P_k''(\sigma i\omega_k)=O(\omega_k^3)
\quad\text{as }k\to\infty.
\]
Since $P_k'(\sigma i\omega_k)\neq 0$ for large $k$, a one-step Newton correction gives
\[
\delta_{k,\sigma}
:=
-\frac{P_k(\sigma i\omega_k)}{P_k'(\sigma i\omega_k)}
=
-\frac{\alpha^2}{2(b-\tau)\omega_k^4}
-\sigma i\,\frac{\alpha^2\tau}{2(b-\tau)\omega_k^3}
+O(\omega_k^{-5}).
\]
Because $P_k''(\sigma i\omega_k)=O(\omega_k^3)$ and $\delta_{k,\sigma}=O(\omega_k^{-3})$, the Taylor remainder is of higher order than the terms retained above. A standard perturbation argument for simple roots (equivalently, one more Newton step or Rouch\'e's theorem on a circle centered at $\sigma i\omega_k+\delta_{k,\sigma}$ of radius $c\omega_k^{-5}$) yields an actual root $s_{k,\sigma}$ satisfying
\[
s_{k,\sigma}
=
\sigma i\omega_k+\delta_{k,\sigma}+O(\omega_k^{-5}).
\]
Since $\omega_k^2=\lambda_k$, this is exactly \eqref{eq:wave-branch-expansion}, and \eqref{eq:wave-branch-real-part} follows immediately by taking real parts. Finally,
\[
b-\tau=\tau\left(\frac{b}{\tau}-1\right)=\tau(c_{\mathrm{mgt}}^2-c_w^2),
\]
which gives \eqref{eq:wave-speed-gap}.
\end{proof}

\begin{remark}
\label{rem:resonant-caution}
Proposition~\ref{prop:modal-obstruction} is stated for fixed $b>\tau$. The expansion is not uniform as $b\downarrow\tau$, because the denominator $b-\tau$ signals a resonant transition in which the wave branch and the high-frequency MGT branch interact on a different scale. The proposition is used here as a quantitative structural explanation of the non-resonant speed-mismatch regime, not as a uniform statement up to the threshold $b=\tau$.
\end{remark}

\section{Coupled potential well}

The exact energy identity \eqref{eq:energy_def} reveals that the natural static variables for the potential well construction are $(u, w)$ rather than $(u, v)$.

\begin{definition}\label{def:potential_well}
Following the classical potential-well framework of Payne--Sattinger \cite{Payne, Sattinger}, for $(u,w)\in H_{0}^{1}(\Omega)\times H_{0}^{1}(\Omega)$ define
\begin{equation*}
Q_{\alpha}(u,w):=\norm{\nabla u}_{2}^{2}+\norm{\nabla w}_{2}^{2}+2\alpha(u,w),
\end{equation*}
and
\begin{align*}
\mathcal{J}_{\alpha}(u,w) &:= \frac{1}{2}Q_{\alpha}(u,w)-\frac{1}{\gamma}\int_{\Omega}|u|^{\gamma}\ln|u|dx+\frac{1}{\gamma^{2}}\norm{u}_{\gamma}^{\gamma}, \\
\mathcal{I}_{\alpha}(u,w) &:= Q_{\alpha}(u,w)-\int_{\Omega}|u|^{\gamma}\ln|u|dx.
\end{align*}
We also set
\begin{align*}
\mathcal{N}_{\alpha} &:= \{(u,w)\ne(0,0):\mathcal{I}_{\alpha}(u,w)=0\}, \\
d_{\alpha} &:= \inf_{(u,w)\in\mathcal{N}_{\alpha}}\mathcal{J}_{\alpha}(u,w).
\end{align*}
Finally, we define the stable set by
\begin{equation*}
\mathcal{W}_{\alpha}:=\{(u,w):\mathcal{J}_{\alpha}(u,w)<d_{\alpha}, \; \mathcal{I}_{\alpha}(u,w)>0\}\cup\{(0,0)\}.
\end{equation*}
\end{definition}

\begin{lemma}[Coercivity of $Q_{\alpha}$]\label{lem:coercivity}
Under \eqref{eq:params},
\begin{equation*}
Q_{\alpha}(u,w)\ge\Bigl(1-\frac{|\alpha|}{\lambda_{1}}\Bigr)\Bigl(\norm{\nabla u}_{2}^{2}+\norm{\nabla w}_{2}^{2}\Bigr)
\end{equation*}
for all $(u,w)\in H_{0}^{1}(\Omega)\times H_{0}^{1}(\Omega)$.
\end{lemma}

\begin{proof}
By Young's inequality, the Poincar\'e inequality, and the definition of $\lambda_{1}$,
\begin{equation*}
2|\alpha||(u,w)|\le|\alpha|(\norm{u}_{2}^{2}+\norm{w}_{2}^{2})\le\frac{|\alpha|}{\lambda_{1}}(\norm{\nabla u}_{2}^{2}+\norm{\nabla w}_{2}^{2}).
\end{equation*}
Hence
\begin{equation*}
Q_{\alpha}(u,w)\ge\Bigl(1-\frac{|\alpha|}{\lambda_{1}}\Bigr)\Bigl(\norm{\nabla u}_{2}^{2}+\norm{\nabla w}_{2}^{2}\Bigr).
\end{equation*}
\end{proof}

\begin{lemma}[Useful algebraic identity]\label{lem:algebraic}
For every $(u,w)\in H_{0}^{1}(\Omega)\times H_{0}^{1}(\Omega)$,
\begin{equation}\label{eq:algebraic}
\mathcal{J}_{\alpha}(u,w)=\frac{\gamma-2}{2\gamma}Q_{\alpha}(u,w)+\frac{1}{\gamma}\mathcal{I}_{\alpha}(u,w)+\frac{1}{\gamma^{2}}\norm{u}_{\gamma}^{\gamma}.
\end{equation}
\end{lemma}

\begin{proof}
Expand the right-hand side and collect the terms involving $Q_{\alpha}$ and $\int_{\Omega}|u|^{\gamma}\ln|u|dx$.
\end{proof}

\begin{lemma}[Logarithmic growth with $\epsilon$-absorption]\label{lem:log_growth}
Choose $\eta>0$ so that $\gamma+\eta<2^{*}:=\frac{2n}{n-2}$. Then, for every $\epsilon>0$, there exists $C_{\epsilon}>0$ such that
\begin{equation}\label{eq:log_est1}
|s|^{\gamma}|\ln|s||\le\epsilon|s|^{2}+C_{\epsilon}|s|^{\gamma+\eta} \quad \text{for all } s\in\mathbb{R}.
\end{equation}
Consequently,
\begin{equation}\label{eq:log_est2}
\int_{\Omega}|u|^{\gamma}|\ln|u||dx\le\frac{\epsilon}{\lambda_{1}}\norm{\nabla u}_{2}^{2}+C_{S,\eta}C_{\epsilon}\norm{\nabla u}_{2}^{\gamma+\eta} \quad \text{for all } u\in H_{0}^{1}(\Omega),
\end{equation}
where $C_{S,\eta}$ denotes the Sobolev embedding constant associated with $H_{0}^{1}(\Omega)\hookrightarrow L^{\gamma+\eta}(\Omega)$. In particular,
\begin{equation}\label{eq:log_est3}
|f(s)|\le\epsilon|s|+C_{\epsilon}|s|^{\gamma-1+\eta} \quad \text{for all } s\in\mathbb{R},
\end{equation}
and, possibly after increasing $C_{\epsilon}$,
\begin{equation}\label{eq:log_est4}
|F(s)|\le\epsilon|s|^{2}+C_{\epsilon}|s|^{\gamma+\eta} \quad \text{for all } s\in\mathbb{R}.
\end{equation}
\end{lemma}

\begin{proof}
Fix $\epsilon>0$. Since $\gamma>2$, we have
\begin{equation*}
\lim_{s\rightarrow 0}|s|^{\gamma-2}|\ln|s||=0.
\end{equation*}
Hence there exists $\delta\in(0,1)$ such that
\begin{equation*}
|s|^{\gamma-2}|\ln|s||\le\epsilon \quad \text{for all } 0<|s|\le\delta,
\end{equation*}
which implies
\begin{equation*}
|s|^{\gamma}|\ln|s||\le\epsilon|s|^{2} \quad \text{for all } |s|\le\delta.
\end{equation*}
On the complementary region $|s|>\delta$, the logarithm is bounded on $[\delta, 1]$ and satisfies the standard growth estimate $|\ln|s||\le C_{\delta,\eta}|s|^{\eta}$ for $|s|\ge 1$. Therefore there exists $C_{\epsilon}>0$ such that
\begin{equation*}
|s|^{\gamma}|\ln|s||\le C_{\epsilon}|s|^{\gamma+\eta} \quad \text{for all } |s|>\delta.
\end{equation*}
Combining the two regions yields \eqref{eq:log_est1}. Integrating over $\Omega$, applying Poincar\'e's inequality to the quadratic term and the Sobolev embedding to the $L^{\gamma+\eta}$ term, gives \eqref{eq:log_est2}. Dividing \eqref{eq:log_est1} by $|s|$ for $s\ne 0$ and setting the value at $s=0$ by continuity yields \eqref{eq:log_est3}. Finally, \eqref{eq:log_est4} follows from the explicit formula for $F$ together with \eqref{eq:log_est1}, after possibly increasing $C_{\epsilon}$.
\end{proof}

\begin{lemma}[Positivity of the well depth]\label{lem:positivity}
The set $\mathcal{N}_{\alpha}$ is nonempty and
\begin{equation*}
d_{\alpha}>0.
\end{equation*}
\end{lemma}

\begin{proof}
Fix $\varphi\in H_{0}^{1}(\Omega)\setminus\{0\}$ and set $w=0$. By the scaling formula
\begin{equation*}
\mathcal{I}_{\alpha}(\lambda\varphi,0)=\lambda^{2}\norm{\nabla\varphi}_{2}^{2}-\lambda^{\gamma}\int_{\Omega}|\varphi|^{\gamma}(\ln\lambda+\ln|\varphi|)dx,
\end{equation*}
we have $\mathcal{I}_{\alpha}(\lambda\varphi,0)>0$ for $\lambda>0$ sufficiently small, while $\mathcal{I}_{\alpha}(\lambda\varphi,0)<0$ for $\lambda$ sufficiently large. Hence, by continuity, $\mathcal{N}_{\alpha}\ne\emptyset$.

Set
\begin{equation*}
c_{\alpha}:=1-\frac{|\alpha|}{\lambda_{1}}>0.
\end{equation*}
Now let $(u,w)\in\mathcal{N}_{\alpha}$. Then $\mathcal{I}_{\alpha}(u,w)=0$, so
\begin{equation*}
Q_{\alpha}(u,w)=\int_{\Omega}|u|^{\gamma}\ln|u|dx\le\int_{\Omega}|u|^{\gamma}|\ln|u||dx.
\end{equation*}
By Lemma \ref{lem:coercivity},
\begin{equation*}
\norm{\nabla u}_{2}^{2}\le c_{\alpha}^{-1}Q_{\alpha}(u,w).
\end{equation*}
Choose $\epsilon>0$ sufficiently small so that
\begin{equation*}
\frac{\epsilon}{\lambda_{1}c_{\alpha}}\le\frac{1}{2}.
\end{equation*}
Then Lemma \ref{lem:log_growth} gives
\begin{equation*}
Q_{\alpha}(u,w)\le\frac{\epsilon}{\lambda_{1}}\norm{\nabla u}_{2}^{2}+C_{S,\eta}C_{\epsilon}\norm{\nabla u}_{2}^{\gamma+\eta}\le\frac{1}{2}Q_{\alpha}(u,w)+\tilde{C}_{\epsilon}Q_{\alpha}(u,w)^{(\gamma+\eta)/2},
\end{equation*}
where $\tilde{C}_{\epsilon}:=C_{S,\eta}C_{\epsilon}c_{\alpha}^{-(\gamma+\eta)/2}$. Subtracting $\frac{1}{2}Q_{\alpha}(u,w)$ from both sides, we obtain
\begin{equation*}
\frac{1}{2}Q_{\alpha}(u,w)\le\tilde{C}_{\epsilon}Q_{\alpha}(u,w)^{(\gamma+\eta)/2}.
\end{equation*}
Since $(u,w)\ne(0,0)$ implies $Q_{\alpha}(u,w)>0$ and since $\gamma+\eta>2$, this yields
\begin{equation*}
Q_{\alpha}(u,w)\ge c_{0}>0 \quad \text{for all } (u,w)\in\mathcal{N}_{\alpha},
\end{equation*}
for some constant $c_{0}$ depending only on the parameters. Using \eqref{eq:algebraic} and the relation $\mathcal{I}_{\alpha}=0$ on $\mathcal{N}_{\alpha}$, we conclude that
\begin{equation*}
\mathcal{J}_{\alpha}(u,w)=\frac{\gamma-2}{2\gamma}Q_{\alpha}(u,w)+\frac{1}{\gamma^{2}}\norm{u}_{\gamma}^{\gamma}\ge\frac{\gamma-2}{2\gamma}c_{0}.
\end{equation*}
Taking the infimum over $\mathcal{N}_{\alpha}$ proves $d_{\alpha}>0$.
\end{proof}

\begin{lemma}[Local positivity of the Nehari functional near the origin]
\label{lem:local_positive}
There exists $\rho_0>0$ such that, whenever $(u,w)\in H_0^1(\Omega)\times H_0^1(\Omega)$ satisfies
\[
Q_\alpha(u,w)\le \rho_0,
\]
one has
\[
\mathcal I_\alpha(u,w)\ge \frac14 Q_\alpha(u,w).
\]
In particular, if $(u,w)\neq (0,0)$ and $Q_\alpha(u,w)\le \rho_0$, then $\mathcal I_\alpha(u,w)>0$.
\end{lemma}

\begin{proof}
Set $c_\alpha:=1-\frac{|\alpha|}{\lambda_1}>0$. Choose $\varepsilon>0$ so small that
\[
\frac{\varepsilon}{\lambda_1 c_\alpha}\le \frac14.
\]
By Lemma~\ref{lem:log_growth} and Lemma~\ref{lem:coercivity},
\[
\int_\Omega |u|^\gamma|\ln|u||\,dx
\le
\frac{\varepsilon}{\lambda_1}\norm{\grad u}_2^2 + C_{S,\eta}C_\varepsilon \norm{\grad u}_2^{\gamma+\eta}
\le
\frac14 Q_\alpha(u,w)+\widetilde C_\varepsilon Q_\alpha(u,w)^{(\gamma+\eta)/2},
\]
where $\widetilde C_\varepsilon:=C_{S,\eta}C_\varepsilon c_\alpha^{-(\gamma+\eta)/2}$. Therefore
\[
\mathcal I_\alpha(u,w)
=Q_\alpha(u,w)-\int_\Omega |u|^\gamma\ln|u|\,dx
\ge \frac34 Q_\alpha(u,w)-\widetilde C_\varepsilon Q_\alpha(u,w)^{(\gamma+\eta)/2}.
\]
Since $\frac{\gamma+\eta}{2}>1$, we can choose $\rho_0>0$ so small that
\[
\widetilde C_\varepsilon Q^{(\gamma+\eta)/2-1}\le \frac12
\qquad\text{for all }0<Q\le \rho_0.
\]
For such $Q=Q_\alpha(u,w)$, the previous inequality yields
\[
\mathcal I_\alpha(u,w)\ge \frac14 Q_\alpha(u,w).
\]
\end{proof}

\section{Global well-posedness in the stable set}

We first fix the natural phase space
\begin{equation*}
\mathcal{H}=H_{0}^{1}(\Omega)\times L^{2}(\Omega)\times H_{0}^{1}(\Omega)\times L^{2}(\Omega)\times H_{0}^{1}(\Omega),
\end{equation*}
with state variable $Y(t)=(u(t),u_{t}(t),w(t),w_{t}(t),v_{t}(t))$.

\begin{definition}[Weak solution]\label{def:weak_sol}
Let $T>0$. A pair $(u, v)$ is called a weak solution of \eqref{eq:main} on $[0, T]$ if, after setting $w=v+\tau v_{t}$, one has
\begin{align*}
u, w &\in L^{\infty}(0,T;H_{0}^{1}(\Omega))\cap C([0,T];L^{2}(\Omega)), \\
u_{t}, w_{t} &\in L^{\infty}(0,T;L^{2}(\Omega))\cap C_{w}([0,T];L^{2}(\Omega)), \\
u_{tt}, w_{tt} &\in L^{\infty}(0,T;H^{-1}(\Omega)), \\
v_{t} &\in L^{\infty}(0,T;H_{0}^{1}(\Omega))\cap C_{w}([0,T];H_{0}^{1}(\Omega)),
\end{align*}
and, for a.e. $t\in(0,T)$,
\begin{align}
\langle u_{tt}(t),\phi\rangle+(\nabla u(t),\nabla\phi)+\alpha(w(t),\phi) &=(f(u(t)),\phi) \quad \forall\phi\in H_{0}^{1}(\Omega), \label{eq:weak1} \\
\langle w_{tt}(t),\psi\rangle+(\nabla w(t),\nabla\psi)+(b-\tau)(\nabla v_{t}(t),\nabla\psi)+\alpha(u(t),\psi) &=0 \quad \forall\psi\in H_{0}^{1}(\Omega). \label{eq:weak2}
\end{align}
The initial conditions are attained as
\begin{equation*}
u(0)=u_{0}, \; w(0)=w_{0} \text{ in } L^{2}(\Omega),
\end{equation*}
and
\begin{equation*}
u_{t}(0)=u_{1}, \; w_{t}(0)=w_{1} \text{ in the weak } L^{2}(\Omega)\text{-sense}.
\end{equation*}
\end{definition}

\begin{remark}[Temporal continuity]\label{rem:continuity}
By the standard Lions-Magenes theorem \cite[Chap. 1]{Lions}, the regularity from Definition \ref{def:weak_sol} implies
\begin{align*}
u, w &\in C_{w}([0,T];H_{0}^{1}(\Omega))\cap C([0,T];L^{2}(\Omega)), \\
u_{t}, w_{t} &\in C_{w}([0,T];L^{2}(\Omega)).
\end{align*}
Interpolating the strong $L^{2}$-continuity with the uniform $H_{0}^{1}$ bound yields
\begin{equation*}
u, w \in C([0,T];L^{q}(\Omega)) \quad \text{for every } 2\le q<2^{*}.
\end{equation*}
In particular, by \eqref{eq:log_est1}, \eqref{eq:log_est3}, and \eqref{eq:log_est4}, the maps
\begin{equation*}
t\mapsto\int_{\Omega}|u(t)|^{\gamma}\ln|u(t)|dx, \quad t\mapsto\norm{u(t)}_{\gamma}^{\gamma}, \quad t\mapsto\int_{\Omega}F(u(t))dx
\end{equation*}
are continuous on $[0,T]$.
\end{remark}

\begin{lemma}[Compactness and nonlinear convergence]\label{lem:compactness}
Let $T>0$ and let $\{u_{m}\}$ be bounded in $L^{\infty}(0,T;H_{0}^{1}(\Omega))\cap W^{1,\infty}(0,T;L^{2}(\Omega))$. Then, up to a subsequence,
\begin{equation*}
u_{m}\rightarrow u \text{ strongly in } C([0,T]; L^{2}(\Omega))
\end{equation*}
and
\begin{equation*}
u_{m}\rightarrow u \text{ strongly in } L^{2}(0,T;L^{q}(\Omega)) \text{ for every } q<2^{*}.
\end{equation*}
Assume, in addition, that the bound \eqref{eq:log_est3} holds for some $\eta>0$ such that
\begin{equation*}
\gamma-1+\eta<\frac{n+2}{n-2}.
\end{equation*}
Then
\begin{equation*}
f(u_{m})\rightarrow f(u) \text{ strongly in } L^{1}((0,T)\times\Omega),
\end{equation*}
and, after extraction,
\begin{equation*}
f(u_m)\rightharpoonup^{*} f(u)
\quad\text{in }L^\infty(0,T;H^{-1}(\Omega)).
\end{equation*}
\end{lemma}

\begin{proof}
The compact embedding $H_{0}^{1}(\Omega)\Subset L^{2}(\Omega)$ and Simon's compactness theorem \cite{Simon} yield the strong convergence in $C([0,T];L^{2}(\Omega))$. Interpolating this convergence with the uniform $L^{\infty}(0,T;L^{2^{*}}(\Omega))$ bound implies the strong convergence in $L^{2}(0,T;L^{q}(\Omega))$ for every $q<2^{*}$. In particular, after passing to a subsequence, $u_{m}\rightarrow u$ almost everywhere on $(0,T)\times\Omega$.

Next set $p:=\frac{2n}{n+2}>1$. Since $\gamma-1+\eta\le\frac{n+2}{n-2}$, the Sobolev embedding gives a uniform bound for $u_{m}$ in $L^{\infty}(0,T;L^{p(\gamma-1+\eta)}(\Omega))$. Using \eqref{eq:log_est3}, we infer that
\begin{equation*}
\norm{f(u_{m})}_{L^{\infty}(0,T;L^{p}(\Omega))}\le C,
\end{equation*}
with $C$ independent of $m$. Hence $\{f(u_{m})\}$ is uniformly integrable in $L^{1}((0,T)\times\Omega)$. Since $f(u_{m})\rightarrow f(u)$ almost everywhere, Vitali's theorem yields
\begin{equation*}
f(u_{m})\rightarrow f(u) \text{ strongly in } L^{1}((0,T)\times\Omega).
\end{equation*}
The weak-* convergence in $L^{\infty}(0,T;H^{-1}(\Omega))$ follows from the uniform bound and the uniqueness of the distributional limit.
\end{proof}

\begin{lemma}[Energy identity for weak solutions]\label{lem:energy_weak}
Let $(u, v)$ be a weak solution of \eqref{eq:main} on $[0,T]$ in the sense of Definition \ref{def:weak_sol}. Then $t\mapsto\mathcal{E}(t)$ belongs to $AC([0,T])$ and, for every $0\le s\le t\le T$,
\begin{equation}\label{eq:energy_weak}
\mathcal{E}(t)+(b-\tau)\int_{s}^{t}\norm{\nabla v_{\xi}(\xi)}_{2}^{2}d\xi=\mathcal{E}(s).
\end{equation}
\end{lemma}

\begin{proof}
For $h>0$, denote by $g^{h}$ the Steklov average in time,
\begin{equation*}
g^{h}(t):=\frac{1}{h}\int_{t}^{t+h}g(\xi)d\xi, \quad 0\le t\le T-h.
\end{equation*}
Applying the averaging operator to \eqref{eq:weak1}-\eqref{eq:weak2}, choosing $(u_{t})^{h}$ and $(w_{t})^{h}$ as test functions, and integrating over $(s,t)\subset[0,T-h]$, we obtain the regularized identity
\begin{equation*}
\mathcal{E}_{h}(t)-\mathcal{E}_{h}(s)+(b-\tau)\int_{s}^{t}\norm{\nabla (v_{t})^{h}(\xi)}_{2}^{2}d\xi=R_{h}(s,t),
\end{equation*}
where $\mathcal{E}_{h}$ is the natural regularized energy and $R_{h}(s,t)\rightarrow 0$ as $h\rightarrow 0$. The convergence of the quadratic terms follows from the weak continuity recorded in Remark \ref{rem:continuity}, while the potential term is handled by \eqref{eq:log_est4} together with the strong continuity of $u$ in $L^{q}(\Omega)$ for every $q<2^{*}$. Passing to the limit as $h\rightarrow 0$ yields \eqref{eq:energy_weak} first for Lebesgue points and then, by absolute continuity, for all $0\le s\le t\le T$.
\end{proof}

\begin{lemma}[Strong continuity in the phase space]
\label{lem:strong_phase}
Let $(u,v)$ be a weak solution of \eqref{eq:main} on $[0,T]$ in the sense of Definition~\ref{def:weak_sol}, and set
\[
Y(t):=(u(t),u_t(t),w(t),w_t(t),v_t(t)).
\]
Then
\[
Y\in C([0,T];\mathcal{H}).
\]
Consequently, the maps
\[
t\longmapsto Q_\alpha(u(t),w(t)),
\qquad
t\longmapsto \mathcal I_\alpha(u(t),w(t)),
\qquad
t\longmapsto \mathcal J_\alpha(u(t),w(t))
\]
are continuous on $[0,T]$.
\end{lemma}

\begin{proof}
By Remark~\ref{rem:continuity}, one already has $Y\in C_w([0,T];\mathcal{H})$. Equip $\mathcal{H}$ with the equivalent Hilbert norm
\[
\|Y(t)\|_*^2:=\norm{\grad u(t)}_2^2+\norm{u_t(t)}_2^2+\norm{\grad w(t)}_2^2+\norm{w_t(t)}_2^2+\tau(b-\tau)\norm{\grad v_t(t)}_2^2.
\]
Using the definition of $\mathcal E(t)$, we can rewrite this norm as
\[
\|Y(t)\|_*^2=2\mathcal E(t)+2\int_\Omega F(u(t)) dx-2\alpha(u(t), w(t)).
\]
The energy term is continuous by Lemma~\ref{lem:energy_weak}, the potential term is continuous by Remark~\ref{rem:continuity}, and the coupling term is continuous because $u,w\in C([0,T];L^2(\Omega))$. Hence $t\mapsto \|Y(t)\|_*$ is continuous on $[0,T]$. Since $Y$ is weakly continuous in the Hilbert space $(\mathcal{H},\|\cdot\|_*)$ and its norm is continuous, the standard Hilbert space argument gives strong continuity of $Y$ in the equivalent norm $\|\cdot\|_*$. Therefore $Y\in C([0,T];\mathcal{H})$ with the original norm as well. The continuity of $Q_\alpha$, $\mathcal I_\alpha$, and $\mathcal J_\alpha$ now follows from the strong continuity of $(u,w)$ in $H_0^1(\Omega)\times H_0^1(\Omega)$ together with Remark~\ref{rem:continuity}.
\end{proof}

\begin{theorem}[Global weak solutions below the well depth]\label{thm:global_existence}
Assume \eqref{eq:params}. Let
\begin{equation*}
u_{0}\in H_{0}^{1}(\Omega), \; u_{1}\in L^{2}(\Omega), \; v_{0}\in H_{0}^{1}(\Omega), \; v_{1}\in H_{0}^{1}(\Omega), \; v_{2}\in L^{2}(\Omega),
\end{equation*}
and define
\begin{equation*}
w_{0}:=v_{0}+\tau v_{1}, \quad w_{1}:=v_{1}+\tau v_{2}.
\end{equation*}
If
\begin{equation*}
\mathcal{E}(0)<d_{\alpha} \quad \text{and} \quad \mathcal{I}_{\alpha}(u_{0},w_{0})>0,
\end{equation*}
then \eqref{eq:main} admits a unique global weak solution in the sense of Definition \ref{def:weak_sol}.  Moreover, the associated state
\[
Y(t)=\bigl(u(t),u_t(t),w(t),w_t(t),v_t(t)\bigr)
\]
belongs to $C([0,\infty);\mathcal H)$. In particular,
\[
u,w\in L^\infty(0,\infty;H_0^1(\Omega))\cap C([0,\infty);H_0^1(\Omega)),
\]
\[
u_t,w_t\in L^\infty(0,\infty;L^2(\Omega))\cap C([0,\infty);L^2(\Omega)),
\]
\[
v_t\in L^\infty(0,\infty;H_0^1(\Omega))\cap C([0,\infty);H_0^1(\Omega)).
\]
Furthermore:
\begin{itemize}
\item[(i)] the exact energy identity \eqref{eq:energy_dissipation} holds for all $t\ge 0$;
\item[(ii)] the configuration pair $(u(t),w(t))$ belongs to $\mathcal{W}_{\alpha}$ for every $t\ge 0$;
\item[(iii)] there exists $C>0$, depending only on the initial data and the parameters, such that
\begin{equation*}
\norm{u_{t}(t)}_{2}^{2}+\norm{w_{t}(t)}_{2}^{2}+\norm{\nabla v_{t}(t)}_{2}^{2}+\norm{\nabla u(t)}_{2}^{2}+\norm{\nabla w(t)}_{2}^{2}+\norm{u(t)}_{\gamma}^{\gamma}\le C
\end{equation*}
for all $t\ge 0$.
\end{itemize}
\end{theorem}

\begin{proof}
We split the argument into four steps.

\textbf{Step 1: Galerkin approximation and uniform bounds.} Let $\{e_{k}\}_{k\ge 1}$ be the Dirichlet eigenfunctions of $-\Delta$ and let $P_{m}$ denote the orthogonal projection onto span$\{e_{1},\dots,e_{m}\}$. Define projected initial data
\begin{equation*}
u_{0m}:=P_{m}u_{0}, \quad u_{1m}:=P_{m}u_{1}, \quad v_{0m}:=P_{m}v_{0}, \quad v_{1m}:=P_{m}v_{1}, \quad v_{2m}:=P_{m}v_{2},
\end{equation*}
and set $w_{0m}:=v_{0m}+\tau v_{1m}$, $w_{1m}:=v_{1m}+\tau v_{2m}$. We seek approximate solutions $(u_{m},v_{m})$ in span$\{e_{1},\dots,e_{m}\}$. Since $f$ is locally Lipschitz on $\mathbb{R}$, the Galerkin coefficients solve a locally well-posed finite-dimensional ODE by the Picard-Lindel\"of theorem. Set $w_{m}:=v_{m}+\tau\partial_{t}v_{m}$. Testing the Galerkin system exactly as in Proposition \ref{prop:energy_identity}, we obtain
\begin{equation*}
\mathcal{E}_{m}(t)+(b-\tau)\int_{0}^{t}\norm{\nabla v_{m,s}(s)}_{2}^{2}ds=\mathcal{E}_{m}(0),
\end{equation*}
where $\mathcal{E}_{m}$ is the approximate counterpart of \eqref{eq:energy_def}. Because
\begin{equation*}
\mathcal{E}_{m}(0)\rightarrow\mathcal{E}(0)<d_{\alpha}, \quad \mathcal{I}_{\alpha}(u_{0m},w_{0m})\rightarrow\mathcal{I}_{\alpha}(u_{0},w_{0})>0,
\end{equation*}
we may assume, for $m$ large enough, that
\begin{equation*}
\mathcal{E}_{m}(0)<d_{\alpha}, \quad \mathcal{I}_{\alpha}(u_{0m},w_{0m})>0.
\end{equation*}
Since the Galerkin trajectories are smooth, the standard first-contact argument yields
\begin{equation*}
\mathcal{I}_{\alpha}(u_{m}(t),w_{m}(t))>0, \quad \mathcal{J}_{\alpha}(u_{m}(t),w_{m}(t))<d_{\alpha} \quad \text{for all } t>0.
\end{equation*}
Applying Lemmas \ref{lem:algebraic} and \ref{lem:coercivity}, we find a constant $C$, independent of $m$ and $t$, such that
\begin{equation*}
\norm{u_{m,t}(t)}_{2}^{2}+\norm{w_{m,t}(t)}_{2}^{2}+\norm{\nabla v_{m,t}(t)}_{2}^{2}+\norm{\nabla u_{m}(t)}_{2}^{2}+\norm{\nabla w_{m}(t)}_{2}^{2}+\norm{u_{m}(t)}_{\gamma}^{\gamma}\le C.
\end{equation*}
Moreover, using the equations together with \eqref{eq:log_est3}, we obtain on every finite interval $[0, T]$,
\begin{equation*}
u_{m,tt}=\Delta u_{m}-\alpha w_{m}+f(u_{m}), \quad w_{m,tt}=\Delta w_{m}+(b-\tau)\Delta v_{m,t}-\alpha u_{m},
\end{equation*}
hence $u_{m,tt}, w_{m,tt}$ are bounded in $L^{\infty}(0,T;H^{-1}(\Omega))$. Therefore the Galerkin solutions extend globally in time.

\textbf{Step 2: Passage to the limit and invariance of the stable set.} Fix $T>0$. By the uniform bounds from Step 1 and Lemma \ref{lem:compactness}, up to a subsequence,
\begin{align*}
u_{m}\rightarrow u, \quad w_{m}\rightarrow w \quad &\text{strongly in } C([0,T];L^{2}(\Omega))\cap L^{2}(0,T;L^{q}(\Omega)) \text{ for every } q<2^{*}, \\
u_{m}\rightharpoonup^{*} u, \quad w_{m}\rightharpoonup^{*} w \quad &\text{in } L^{\infty}(0,T;H_{0}^{1}(\Omega)), \\
u_{m,t}\rightharpoonup^{*} u_{t}, \quad w_{m,t}\rightharpoonup^{*} w_{t} \quad &\text{in } L^{\infty}(0,T;L^{2}(\Omega)), \\
v_{m,t}\rightharpoonup^{*} v_{t} \quad &\text{in } L^{\infty}(0,T;H_{0}^{1}(\Omega)), \\
f(u_{m})\rightarrow f(u) \quad &\text{in } L^{1}((0,T)\times\Omega).
\end{align*}
Passing to the limit in the Galerkin formulation yields a weak solution in the sense of Definition \ref{def:weak_sol}. By Lemmas~\ref{lem:energy_weak} and \ref{lem:strong_phase}, the associated state $Y$ belongs to $C([0,T];\mathcal H)$. In particular,
\[
t\longmapsto \mathcal I_\alpha(u(t),w(t)),
\qquad
t\longmapsto \mathcal J_\alpha(u(t),w(t))
\]
are continuous on $[0,T]$.

Moreover, the energy identity gives
\[
\mathcal J_\alpha(u(t),w(t))
\le \mathcal E(t)
\le \mathcal E(0)<d_\alpha
\qquad\text{for all }t\in[0,T].
\]
We claim that
\[
\mathcal I_\alpha(u(t),w(t))>0
\qquad\text{for all }t\in[0,T].
\]
Suppose, on the contrary, that the set
\[
\Sigma:=\{t\in[0,T]:\ \mathcal I_\alpha(u(t),w(t))<0\}
\]
is nonempty, and let $t_*:=\inf\Sigma$. Since $\mathcal I_\alpha(u_0,w_0)>0$ and $t\mapsto \mathcal I_\alpha(u(t),w(t))$ is continuous, one has $0<t_*\le T$ and
\[
\mathcal I_\alpha(u(t_*),w(t_*))=0.
\]
If $(u(t_*),w(t_*))\neq(0,0)$, then $(u(t_*),w(t_*))\in\mathcal N_\alpha$, and therefore
\[
\mathcal J_\alpha(u(t_*),w(t_*))\ge d_\alpha,
\]
which contradicts the strict bound above. If instead $(u(t_*),w(t_*))=(0,0)$, then Lemma~\ref{lem:local_positive} and the continuity of $t\mapsto Q_\alpha(u(t),w(t))$ imply that there exists $\delta>0$ such that
\[
Q_\alpha(u(t),w(t))\le \rho_0
\qquad\text{for all }t\in[t_*,\min\{T,t_*+\delta\}],
\]
and hence
\[
\mathcal I_\alpha(u(t), w(t)) \ge 0 \quad \text{for all } t \in [t_*, \min\{T,t_*+\delta\}],
\]
with strict positivity whenever \((u(t),w(t))\neq(0,0)\). In particular, no point of
\((t_*, \min\{T,t_*+\delta\}]\) belongs to \(\Sigma\), contradicting the definition of
\(t_*\) as the infimum of \(\Sigma\). Therefore $\Sigma=\varnothing$, so $\mathcal I_\alpha(u(t),w(t))>0$ for every $t\in[0,T]$. Consequently, $(u(t),w(t))\in\mathcal W_\alpha$ on $[0,T]$, and since $T>0$ is arbitrary this remains true for all $t\ge0$.

\textbf{Step 3: Exact energy identity and uniform control.} By Lemma \ref{lem:energy_weak}, the weak solution satisfies the exact energy identity \eqref{eq:energy_dissipation}. Since $(u(t),w(t))\in\mathcal{W}_{\alpha}$ and $\mathcal{I}_{\alpha}(u(t),w(t))\ge 0$ for every $t\ge 0$, Lemma \ref{lem:algebraic} implies
\begin{equation*}
\mathcal{J}_{\alpha}(u(t),w(t))\ge\frac{\gamma-2}{2\gamma}Q_{\alpha}(u(t),w(t))+\frac{1}{\gamma^{2}}\norm{u(t)}_{\gamma}^{\gamma}.
\end{equation*}
Combining this inequality with Lemma \ref{lem:coercivity} and the bound $\mathcal{E}(t)\le\mathcal{E}(0)$ yields
\begin{equation*}
\norm{u_{t}(t)}_{2}^{2}+\norm{w_{t}(t)}_{2}^{2}+\norm{\nabla v_{t}(t)}_{2}^{2}+\norm{\nabla u(t)}_{2}^{2}+\norm{\nabla w(t)}_{2}^{2}+\norm{u(t)}_{\gamma}^{\gamma}\le C
\end{equation*}
for all $t\ge 0$. This proves the asserted a priori estimate.

\textbf{Step 4: Uniqueness and continuous dependence.} Let $(u^{1},v^{1})$ and $(u^{2},v^{2})$ be two weak solutions with initial data in the class of the theorem, and set
\begin{equation*}
z=u^{1}-u^{2}, \quad r=v^{1}-v^{2}, \quad y=w^{1}-w^{2}=r+\tau r_{t}.
\end{equation*}
Then
\begin{equation}\label{eq:diff}
\begin{cases}
z_{tt}-\Delta z+\alpha y=f(u^{1})-f(u^{2}), \\
y_{tt}-\Delta y-(b-\tau)\Delta r_{t}+\alpha z=0.
\end{cases}
\end{equation}
Fix $\mu>0$ so small that
\begin{equation*}
\gamma-2+\mu<\frac{2}{n-2}.
\end{equation*}
For $s\ne 0$,
\begin{equation*}
f^{\prime}(s)=|s|^{\gamma-2}((\gamma-1)\ln|s|+1),
\end{equation*}
hence
\begin{equation*}
|f^{\prime}(s)|\le C_{\mu}(1+|s|^{\gamma-2+\mu}).
\end{equation*}
By the mean value theorem,
\begin{equation}\label{eq:f_diff}
|f(a)-f(b)|\le C_{\mu}(1+|a|^{\gamma-2+\mu}+|b|^{\gamma-2+\mu})|a-b|.
\end{equation}
Since the weak regularity does not allow us to test \eqref{eq:diff} directly by $(z_{t},y_{t})$, we first apply Steklov averaging in time, use the regularized pair $((z_{t})^{h},(y_{t})^{h})$ as test functions, and then pass to the limit $h\rightarrow 0$. This yields the difference energy identity
\begin{equation*}
\frac{d}{dt}\mathcal{Z}(t)+(b-\tau)\norm{\nabla r_{t}(t)}_{2}^{2}=\int_{\Omega}(f(u^{1})-f(u^{2}))z_{t}dx,
\end{equation*}
where
\begin{equation*}
\mathcal{Z}(t):=\frac{1}{2}\norm{z_{t}}_{2}^{2}+\frac{1}{2}\norm{y_{t}}_{2}^{2}+\frac{\tau(b-\tau)}{2}\norm{\nabla r_{t}}_{2}^{2}+\frac{1}{2}\norm{\nabla z}_{2}^{2}+\frac{1}{2}\norm{\nabla y}_{2}^{2}+\alpha(z,y).
\end{equation*}
By Lemma \ref{lem:coercivity}, $\mathcal{Z}(t)$ is equivalent to
\begin{equation*}
\norm{z_{t}}_{2}^{2}+\norm{y_{t}}_{2}^{2}+\norm{\nabla r_{t}}_{2}^{2}+\norm{\nabla z}_{2}^{2}+\norm{\nabla y}_{2}^{2}.
\end{equation*}
Using \eqref{eq:f_diff}, H\"older's inequality with exponents $n$, $\frac{2n}{n-2}$, and $2$, together with the Sobolev embedding, we infer
\begin{align*}
\int_{\Omega}|f(u^{1})-f(u^{2})||z_{t}|dx &\le C(1+\norm{u^{1}}_{n(\gamma-2+\mu)}^{\gamma-2+\mu}+\norm{u^{2}}_{n(\gamma-2+\mu)}^{\gamma-2+\mu})\norm{z}_{\frac{2n}{n-2}}\norm{z_{t}}_{2} \\
&\le C(\norm{\nabla z}_{2}^{2}+\norm{z_{t}}_{2}^{2})\le C\mathcal{Z}(t).
\end{align*}
Therefore $\mathcal{Z}^{\prime}(t)\le C\mathcal{Z}(t)$. Gronwall's lemma yields uniqueness when the initial data coincide. More generally,
\begin{equation*}
\mathcal{Z}(t)\le e^{Ct}\mathcal{Z}(0) \quad \text{for all } t\ge 0,
\end{equation*}
which provides continuous dependence on the initial state in the phase space $\mathcal{H}$.
\end{proof}

\section{Conditional asymptotic stability}

The previous theorem provides a global solution theory and a strict Lyapunov functional. To place the asymptotic argument into the standard infinite dimensional framework, we first record the resulting semiflow structure.

\begin{proposition}[Semiflow and Lyapunov structure]\label{prop:semiflow}
For
\[
Y=(u_0,u_1,w_0,w_1,v_1)\in\mathcal H,
\]
define the phase energy
\[
\mathcal E(Y):=
\frac12\|u_1\|_2^2
+\frac12\|w_1\|_2^2
+\frac{\tau(b-\tau)}2\|\nabla v_1\|_2^2
+\frac12\|\nabla u_0\|_2^2
+\frac12\|\nabla w_0\|_2^2
+\alpha(u_0,w_0)
-\int_\Omega F(u_0)\,dx.
\]
Let
\[
\mathcal X_\alpha
:=
\Bigl\{
Y=(u_0,u_1,w_0,w_1,v_1)\in\mathcal H:
\ \mathcal E(Y)<d_\alpha,\ (u_0,w_0)\in\mathcal W_\alpha
\Bigr\}.
\]
Then, for every \(Y_0\in\mathcal X_\alpha\), there exists a unique global weak solution of \eqref{eq:main}, and we denote by \(S(t)Y_0\) the corresponding state at time \(t\ge0\). The family \(\{S(t)\}_{t\ge0}\) is a continuous semiflow on \(\mathcal X_\alpha\). Moreover, for every \(Y_0\in\mathcal X_\alpha\) and every \(t\ge0\),
\[
\mathcal E(S(t)Y_0)
+
(b-\tau)\int_0^t \|\nabla v_s(s)\|_2^2\,ds
=
\mathcal E(Y_0),
\]
so \(\mathcal E\) is a Lyapunov functional on \(\mathcal X_\alpha\). Consequently, every relatively compact orbit in \(\mathcal X_\alpha\) has a nonempty compact invariant \(\omega\)-limit set, and Hale's version of LaSalle's invariance principle applies; see \cite[Chap.~1]{Hale}.
\end{proposition}

\begin{proof}
If \((u_0,w_0)\neq(0,0)\), then \((u_0,w_0)\in\mathcal W_\alpha\) means
\[
\mathcal J_\alpha(u_0,w_0)<d_\alpha,
\qquad
\mathcal I_\alpha(u_0,w_0)>0,
\]
and therefore Theorem \ref{thm:global_existence} provides a unique global weak solution.

If \((u_0,w_0)=(0,0)\), then \(\mathcal E(Y_0)<d_\alpha\), and the same Galerkin construction used in the proof of Theorem \ref{thm:global_existence} yields a global weak solution. In this case, Lemma \ref{lem:local_positive} gives local positivity of \(\mathcal I_\alpha\) as soon as the configuration leaves the origin, and the first-contact argument from Step~2 of Theorem \ref{thm:global_existence} shows that
\[
(u(t),w(t))\in\mathcal W_\alpha
\qquad\text{for all }t\ge0.
\]
Uniqueness follows from the same continuous dependence estimate proved in Step~4 of Theorem \ref{thm:global_existence}, whose derivation uses only the difference system and the a priori energy bounds.

The semigroup property
\[
S(t+s)Y_0=S(t)(S(s)Y_0)
\qquad\text{for all }t,s\ge0
\]
follows from uniqueness. The continuity of \(S(t)\) with respect to the initial state follows from the continuous dependence estimate of Step~4 in Theorem \ref{thm:global_existence}. The Lyapunov property is exactly the energy identity from Theorem \ref{thm:global_existence} (i). The final statement is standard for continuous semiflows with precompact orbits and a Lyapunov functional; see Hale \cite[Chap.~1]{Hale}.
\end{proof}

\begin{theorem}[Conditional asymptotic stability]\label{thm:asymptotic}
Assume \eqref{eq:params} and, in addition,
\begin{equation*}
0<|\alpha|<\lambda_{1}.
\end{equation*}
Let $(u,v)$ be the global weak solution given by Theorem \ref{thm:global_existence}. If the trajectory
\begin{equation*}
\{Y(t):t\ge 0\}\subset\mathcal{H}
\end{equation*}
is relatively compact in $\mathcal{H}$, then
\begin{equation*}
\lim_{t\rightarrow\infty}\mathcal{E}(t)=0.
\end{equation*}
Equivalently,
\begin{equation*}
\lim_{t\rightarrow\infty}(\norm{u_{t}(t)}_{2}^{2}+\norm{w_{t}(t)}_{2}^{2}+\norm{\nabla v_{t}(t)}_{2}^{2}+\norm{\nabla u(t)}_{2}^{2}+\norm{\nabla w(t)}_{2}^{2})=0,
\end{equation*}
and $(u(t),w(t))\rightarrow(0,0)$ in $H_{0}^{1}(\Omega)\times H_{0}^{1}(\Omega)$.
\end{theorem}

\begin{proof}
Let $Y_{0}=(u_{0},u_{1},w_{0},w_{1},v_{1})\in\mathcal{X}_{\alpha}\setminus\{0_{\mathcal H}\}$ be the initial state, and let \(Y(t)=S(t)Y_0\) be the corresponding global trajectory. By Proposition \ref{prop:semiflow}, the $\omega$-limit set $\omega(Y_{0})$ is nonempty, compact, and invariant. Moreover, by Hale's invariance principle, it is contained in the largest invariant subset of
\begin{equation*}
\mathcal{M}:=\{Y\in\mathcal{X}_{\alpha}:\norm{\nabla v_{t}}_{2}=0\}.
\end{equation*}
Let $\bar{Y}(t)=(\bar{u},\bar{u}_{t},\bar{w},\bar{w}_{t},\bar{v}_{t})(t)$ be a complete trajectory contained in this invariant subset. Then, for all $t\in\mathbb{R}$,
\begin{equation*}
\norm{\nabla\bar{v}_{t}(t)}_{2}=0.
\end{equation*}
Because $\bar{v}_{t}$ satisfies homogeneous Dirichlet boundary conditions, Poincar\'e's inequality implies
\begin{equation*}
\bar{v}_{t}(t)\equiv 0 \quad \text{for all } t\in\mathbb{R}.
\end{equation*}
Hence also $\bar{v}_{tt}\equiv 0$, $\bar{v}_{ttt}\equiv 0$, and
\begin{equation*}
\bar{w}_{t}=\bar{v}_{t}+\tau\bar{v}_{tt}\equiv 0.
\end{equation*}
The second equation of \eqref{eq:main} therefore reduces to
\begin{equation*}
-\Delta\bar{v}+\alpha\bar{u}=0.
\end{equation*}
Since $\bar{v}_{t}\equiv 0$, the function $\bar{v}$ is independent of time, and so is its spatial Laplacian $-\Delta\bar{v}$. The preceding relation then shows that $\bar{u}$ is time-independent. Because $\alpha\ne 0$, we conclude that
\begin{equation*}
\bar{u}_{t}\equiv 0.
\end{equation*}
Thus $\bar{u}_{tt}\equiv 0$, and $\bar{w}\equiv\bar{v}$ is likewise independent of time. Substituting into \eqref{eq:augmented}, we obtain the stationary elliptic system
\begin{equation}\label{eq:stationary}
\begin{cases}
-\Delta\bar{u}+\alpha\bar{w}=f(\bar{u}), \\
-\Delta\bar{w}+\alpha\bar{u}=0.
\end{cases}
\end{equation}
Testing \eqref{eq:stationary} respectively by $\bar{u}$ and $\bar{w}$ and summing, we find
\begin{equation*}
Q_{\alpha}(\bar{u},\bar{w})=\int_{\Omega}|\bar{u}|^{\gamma}\ln|\bar{u}|dx,
\end{equation*}
that is, $\mathcal{I}_{\alpha}(\bar{u},\bar{w})=0$.

Let
\begin{equation*}
\mathcal{E}_{\infty}:=\lim_{t\rightarrow\infty}\mathcal{E}(t).
\end{equation*}
Since $\mathcal{E}$ is nonincreasing and bounded below by 0,
\begin{equation*}
0\le\mathcal{E}_{\infty}\le\mathcal{E}(0)<d_{\alpha}.
\end{equation*}
All kinetic terms vanish on $\omega(Y_{0})$, so
\begin{equation*}
\mathcal{J}_{\alpha}(\bar{u},\bar{w})=\mathcal{E}_{\infty}<d_{\alpha}.
\end{equation*}
If $(\bar{u},\bar{w})\ne(0,0)$, then $(\bar{u},\bar{w})\in\mathcal{N}_{\alpha}$ and therefore
\begin{equation*}
\mathcal{J}_{\alpha}(\bar{u},\bar{w})\ge d_{\alpha},
\end{equation*}
which contradicts the previous strict inequality. Hence every complete trajectory contained in the largest invariant subset of $\mathcal{M}$ is trivial, namely,
\begin{equation*}
(\bar{u},\bar{w})=(0,0).
\end{equation*}
By Proposition \ref{prop:semiflow} and Hale's invariance principle, the full trajectory converges to the zero equilibrium. In particular,
\begin{equation*}
\lim_{t\rightarrow\infty}\mathcal{E}(t)=0.
\end{equation*}
\end{proof}

\section{Conclusions}

The central structural contribution of this paper is the introduction of the coupled state variable $w=v+\tau v_{t}$. This change of variables unambiguously identifies the exact coupled energy and the appropriate static potential well geometry without resorting to artificial cross multiplier techniques. It also makes transparent---and Proposition~\ref{prop:modal-obstruction} quantifies at the linearized level---why, for a linearly coupled system with zero order interaction and mismatched propagation speeds, the available dissipation does not readily yield a uniform decay statement. In this way, the analysis is redirected toward a rigorous qualitative stabilization framework based on the exact energy identity and the coupled potential well.



\end{document}